\documentclass[10pt,a4paper]{article}

\usepackage[final]{optional}

\usepackage{amsfonts, amsmath, wasysym}

\usepackage{amssymb,amsthm,
paralist
}

\usepackage{
latexsym,
}

\usepackage[usenames]{color}

\usepackage{url}

\definecolor{darkgreen}{rgb}{0,0.5,0}
\definecolor{darkred}{rgb}{0.7,0,0}
\usepackage[colorlinks, 
citecolor=darkgreen, linkcolor=darkred
]{hyperref}

\theoremstyle{plain}
\newtheorem{lemma}{Lemma}[section]
\newtheorem{thm}[lemma]{Theorem}
\newtheorem{prop}[lemma]{Proposition}
\newtheorem{cor}[lemma]{Corollary}

\theoremstyle{definition}
\newtheorem{defn}[lemma]{Definition}

\newtheorem{rmk}[lemma]{Remark}

\numberwithin{equation}{section}

\newcommand{\m}{\ensuremath{{\cal M}}}
\newcommand{\n}{\ensuremath{{\cal N}}}

\newcommand{\pl}[2]{{\frac{\partial #1}{\partial #2}}}

\newcommand{\al}{\alpha}
\newcommand{\be}{\beta}
\newcommand{\ga}{\gamma}

\newcommand{\de}{\delta}

\newcommand{\Om}{\Omega}

\newcommand{\la}{\lambda}

\newcommand{\si}{\sigma}

\renewcommand{\th}{\theta}

\newcommand{\vph}{\varphi}
\newcommand{\ep}{\varepsilon}

\newcommand{\R}{\ensuremath{{\mathbb R}}}

\newcommand{\C}{\ensuremath{{\mathbb C}}}

\newcommand{\downto}{\downarrow}

\newcommand{\lap}{\Delta}

\newcommand{\grad}{\nabla}

\DeclareMathOperator{\Vol}{Vol}

\def\blbox{\quad \vrule height7.5pt width4.17pt depth0pt}

\newcommand{\beq}{\begin{equation}}
\newcommand{\eeq}{\end{equation}}
\newcommand{\beqa}{\begin{equation}\begin{aligned}}
\newcommand{\eeqa}{\end{aligned}\end{equation}}
\newcommand{\brmk}{\begin{rmk}}
\newcommand{\ermk}{\end{rmk}}
\newcommand{\partref}[1]{\hbox{(\csname @roman\endcsname{\ref{#1}})}}
\newcommand{\half}{\frac{1}{2}}

\newcommand{\cmt}[1]{\opt{draft}{\textcolor[rgb]{0.5,0,0}{
$\LHD$ #1 $\RHD$\marginpar{\blbox}}}}

\newcommand{\Rm}{{\mathrm{Rm}}}
\newcommand{\Ric}{{\mathrm{Ric}}}

\newcommand*\dz{\ensuremath{\mathrm{d}z}}

\usepackage{soul}

\title{{
\bf
Uniqueness and nonuniqueness for Ricci flow on surfaces:
Reverse cusp singularities
}
\\ 
\cmt{DRAFT with comments}
}
\author{Peter M. Topping}
\date{\today}

\begin{document}

\maketitle
\parskip=10pt

\begin{abstract}
We extend the notion of what it means for a complete 
Ricci flow to have a given initial metric, and consider
the resulting well-posedness issues that arise in the
2D case.
On one hand we construct examples of nonuniqueness
by showing that surfaces with cusps can evolve either
by keeping the cusps or by contracting them.
On the other hand, by adding a noncollapsedness assumption
for the initial metric, we establish a uniqueness result.
\end{abstract}

\section{Introduction}
\label{intro}

A complete Ricci flow $(\m,g(t))$ is a smooth family of complete 
Riemannian metrics on a manifold $\m$, for $t$ within some
interval in $\R$, which satisfies Hamilton's nonlinear PDE
\beq
\label{RFeq}
\pl{}{t}g(t)=-2\,\Ric[g(t)].
\eeq
In the case that $\m$ is two-dimensional, the flow preserves
the conformal class of the metric and can be written
$$\pl{}{t}g(t)=-2K g(t),$$
where $K$ is the Gauss curvature of $g(t)$.
In this case, we may take local isothermal coordinates
$x$ and $y$, and write the flow $g(t)=e^{2u}(dx^2+dy^2)$
for some locally-defined scalar time-dependent function $u$ which will
then satisfy the local equation
\beq
\label{2DRFeq}
\pl{u}{t}=e^{-2u}\lap u = -K.
\eeq
Returning to the case that $\m$ is of arbitrary dimension,
Hamilton \cite{ham3D} and Shi \cite{shi} developed an
existence theory for this equation when a complete 
bounded-curvature initial metric $g_0$ was specified.
In other words, they found a complete bounded-curvature
Ricci flow $g(t)$ for $t\in [0,T]$ (some $T>0$) with
$g(0)=g_0$.
Hamilton \cite{ham3D} and Chen-Zhu \cite{chenzhu}
proved that this flow is unique within the class of
complete bounded-curvature Ricci flows.

In this paper we consider existence and particularly uniqueness
issues when we drop the restriction that the complete Ricci flows have
bounded curvature, and generalise the notion of initial metric 
as follows.

\begin{defn}
\label{initconddef}
We say that a complete Ricci flow  $(\m,g(t))$ for $t\in (0,T]$ 
has a complete Riemannian manifold $(\n,g_0)$ as initial condition
if there exists a smooth map $\vph:\n\to\m$, diffeomorphic
onto its image, such that 
$$\vph^*(g(t))\to g_0$$
smoothly locally on $\n$ as $t\downto 0$.
\end{defn}

In practice, we will be interested in the case that
$(\n,g_0)$ has bounded curvature but $g(t)$ is allowed to
have curvature with no uniform upper bound.
In this way, $\m$ and $\n$ may not be diffeomorphic since
parts of $\m$ may be shot out to infinity as $t\downto 0$ 
resulting in a change of topology in the limit.

This generalised notion of initial condition permits some 
new types of solution which do not fit into the classical
framework. In particular, we show that a bounded-curvature
Riemannian surface with a hyperbolic cusp need
not be obliged to flow forwards in time retaining the cusp
(as in Shi's solution) but can add in a point at infinity,
removing the puncture in the surface, and let the cusp contract
in a controlled way. More generally we have:

\begin{thm}
\label{cuspcontractthm}
Suppose $\m$ is a compact Riemann surface and 
$\{p_1,\ldots,p_n\}\subset \m$ 
is a finite set of distinct points. If $g_0$ is a complete, 
bounded-curvature, smooth, conformal metric on $\n:=\m\backslash \{p_1,\ldots,p_n\}$
with strictly negative curvature in a neighbourhood of each point
$p_i$, then there exists a 
Ricci flow $g(t)$ on $\m$
for $t\in (0,T]$ (for some $T>0$) having $(\n,g_0)$ as initial
condition in the sense of Definition \ref{initconddef}. 
We can take the map $\vph$ there to be the natural 
inclusion of $\n$ in $\m$. 

Moreover, the cusps contract logarithmically in the sense
that for some $C<\infty$ and all 
$t\in (0,T]$ sufficiently small, 
we have
\begin{equation}
\label{logdecay}
\frac{1}{C}(-\ln t)\leq diam (\m,g(t)) \leq C(-\ln t).
\end{equation}

Furthermore, the curvature of $g(t)$ is bounded below uniformly
as $t\downto 0$.
\end{thm}

Thus a specific example of nonuniqueness would be when 
the Riemann surface $\m$ is a torus $T^2$, 
we remove one point to give $\n$, and let $g_0$ be the 
unique complete conformal hyperbolic metric on $\n$.
One Ricci flow continuation would be the homothetically
expanding one (which coincides with the solution constructed
by Shi) but another continuation would see the cusp 
contract with the subsequent Ricci flow living on the whole
torus $\m$.

One characteristic of these nonuniqueness examples is 
that the initial condition $(\n,g_0)$ does not have a lower
bound for its injectivity radius, or equivalently that 
one can find unit balls of arbitrarily small area.
In fact, we will see in a corollary to the following theorem 
that this is a necessary condition for nonuniqueness.

\begin{thm}
\label{noncollapsedthm}
Suppose that $(\n,g_0)$ is a complete Riemannian surface 
with bounded curvature which is noncollapsed in the sense
that for some $r_0>0$ we have 
\begin{equation}
\label{noncollapsed}
\Vol_{g_0}(B_{g_0}(x,r_0))\geq \ep>0
\eeq
for all $x\in\n$.
If $(\m,g(t))$ is a complete Ricci flow for $t\in (0,T]$ (some $T>0$)
which has $(\n,g_0)$ as initial condition in the sense of
Definition \eqref{initconddef}, then $(\m,g(t))$ has uniformly
bounded curvature over some time interval $(0,\de]$ 
(some $\de\in (0,T]$). Moreover, the $\vph$ from Definition 
\eqref{initconddef} must be a diffeomorphism (i.e. also surjective)
and $g(t)$ can be extended smoothly down to $t=0$ on the whole
of $\m$ by setting $g(0):=\vph_* g_0$.
\end{thm}

The proof of this theorem uses the work of Chen 
\cite{stronguniqueness}, which in turn uses the work
of Perelman \cite{P1}. It is possible to prove a 
variant of this result which is applicable to Ricci flows
on higher-dimensional manifolds, albeit with slightly 
stronger hypotheses (see \cite{rick_survey}).

\begin{cor}
\label{uniquenesscor}
With $(\n,g_0)$ as in the theorem above, if for $i=1,2$ we have
complete Ricci flows $(\m_i,g_i(t))$ for $t\in (0,T_i]$ (some $T_i>0$)
with $(\n,g_0)$ as initial condition, 
then these two Ricci flows must agree over some nonempty
time interval $t\in (0,\de]$ in the sense that
there exists a diffeomorphism $\psi:\m_1\to\m_2$ with
$\psi^*(g_2(t))=g_1(t)$ for all $t\in (0,\de]$.
\end{cor}

Despite the nonuniqueness implied by Theorem \ref{cuspcontractthm},
that construction throws up a quite different uniqueness issue:
Does there exist more than one flow which does the same job
of contracting the cusps? The next result shows that there does
not.

\begin{thm}
\label{unique_after_all}
In the situation of Theorem \ref{cuspcontractthm} (in which 
$\vph$ is the natural inclusion of $\n$ into $\m$) if
$\tilde g(t)$ is a smooth Ricci flow on $\m$ for
some time interval $t\in (0,\de)$ ($\de\in (0,T]$)
such that $\tilde g(t)\to g_0$ smoothly locally on $\n$ 
as $t\downto 0$ and the Gauss curvature of $\tilde g(t)$ is uniformly 
bounded below, then $\tilde g(t)$ agrees with the flow $g(t)$ constructed in Theorem \ref{cuspcontractthm} for $t\in (0,\de)$.
\end{thm}

Returning to Theorem \ref{cuspcontractthm}, one can ask
at what rate the curvature of $g(t)$ 
must blow up in the limit $t\downto 0$.
General theory tells us that this sort of behaviour cannot
occur if the curvature blows up no faster than $C/t$ (\cite{miles_oneont},
\cite{stronguniqueness}). By analogy with the terminology of Hamilton
for blow up rates \cite{formations}, we might say then that
we have a `Type II(c) singularity' meaning that
$$\limsup_{t\downto 0}\left[t\sup_\m|\Rm(\cdot,t)|\right]=\infty.$$ 
In fact, a rough asymptotic analysis of a contracting cusp
in the rotationally symmetric case, modelled by a hyperbolic
cusp capped off by an appropriately scaled cigar soliton,
suggests that the curvature blows up at a rate 
$C/t^2$.

Finally, we point out that Theorem \ref{cuspcontractthm} 
provides an answer
to Perelman's question \cite[\S 10.3]{P1} of whether the volume
ratio hypothesis is necessary in his pseudolocality theorem:
It is. More elementary examples can also be constructed (\cite{sendai_talk}).

The paper is organised as follows. In Section \ref{noncollsect}
we prove Theorem \ref{noncollapsedthm} and
its Corollary \ref{uniquenesscor}.
In Section \ref{fccsect} we derive a selection of estimates
for metrics on punctured discs, and use them to construct 
useful barriers and prove useful estimates for Ricci flow, 
with the key tool being Lemma \ref{flowupperbdlemma}. 
This technology is 
then used to prove Theorem \ref{cuspcontractthm}. 
Finally, in Section \ref{alt_sect}, we prove the uniqueness
assertion of Theorem \ref{unique_after_all}.

\emph{Acknowledgements:} This work was supported by 
The Leverhulme Trust. Parts of this work were carried
out when the author was visiting the Max Planck Albert
Einstein Institute, Golm, and the Free University, Berlin,
and he would like to thank Gerhard Huisken and Klaus Ecker
for their hospitality. Thanks also to Gregor Giesen for discussions
on the paper \cite{stronguniqueness}.

\section{Noncollapsed initial metrics}
\label{noncollsect}

In this section we prove Theorem \ref{noncollapsedthm} and
its Corollary \ref{uniquenesscor}.
The proof will extend slightly the uniqueness result in
the work of Chen \cite{stronguniqueness}, which appeals
strongly to the remarkable properties of the distance 
function on a Ricci flow discovered by Perelman \cite{P1}
in order to prove the following curvature estimate.

\begin{prop} (Chen \cite[Proposition 3.9]{stronguniqueness},
cf. Perelman \cite[\S 10.3]{P1}.)
\label{chenprop}
Let $\m$ be a surface and $g(t)$ a smooth Ricci flow on $\m$
for $t\in [0,T]$. Suppose that $x_0\in\m$ and $r_0>0$, and 
that $B_{g(t)}(x_0,r_0)\subset\subset\m$ for all $t\in [0,T]$.
If 
$$|R[g(0)]|\leq r_0^{-2}\text{ on }B_{g(0)}(x_0,r_0)
\qquad\text{ and }\qquad
\Vol_{g(0)}(B_{g(0)}(x_0,r_0))\geq v_0 r_0^2$$
for some $v_0>0$, then there exists $\de>0$ depending
on $v_0$ such that
$$|R[g(t)]|\leq 2r_0^{-2}\text{ on }B_{g(t)}\left(x_0,\frac{r_0}{2}\right)$$
for all $t\in [0,T]$ with $t\leq \de r_0^2$.
\end{prop}

\begin{proof} (Theorem \ref{noncollapsedthm}.)
First note that by the Bishop-Gromov comparison theorem,
we may reduce $r_0$ to any smaller positive value and
still have the noncollapsedness condition
\eqref{noncollapsed} for some new, possibly smaller, positive
value of $\ep$.
In particular, by making such a reduction we may assume
also that $|R[g_0]|\leq \half r_0^{-2}$ thoughout $\n$.

We now set $v_0=\frac{\ep}{2}r_0^{-2}$ and attempt to apply 
Proposition \ref{chenprop} to $(\m,g(t))$. Since $(\n,g_0)$ is the 
initial condition for $(\m,g(t))$, we see that for all 
$x_0\in\m$, and sufficiently small $t_0>0$ (depending on $x_0$) 
we have
$$|R[g(t_0)]|\leq r_0^{-2}\text{ on }B_{g(t_0)}(x_0,r_0)
\qquad\text{ and }\qquad
\Vol_{g(t_0)}(B_{g(t_0)}(x_0,r_0))\geq \frac{\ep}{2}=v_0 r_0^2.$$
Keeping in mind that we may take $t_0>0$ arbitrarily small, 
Proposition \ref{chenprop} then implies that 
$$|R[g(t)]|\leq 2r_0^{-2}\text{ on }B_{g(t)}\left(x_0,\frac{r_0}{2}\right)$$
for all $t\in (0,T]$ with $t\leq \de r_0^2$.
Since $x_0$ was arbitrary, we have established the required 
uniform curvature bound for $g(t)$.

The uniform curvature bound is then enough to force $\vph$
to be a diffeomorphism. Indeed, if we suppose that 
$\vph$ is not surjective, then we can pick $y\in\m$ outside 
its image. We can then take any
smooth immersed curve $\ga:[0,1]\to\m$ so that $\ga(0)$ lies
within the image of $\vph$ and $\ga(1)=y$.
By truncating and reparametrising the curve, and adjusting $y$, 
we may assume that $\ga(s)$ lies in the image of $\vph$ 
precisely for $s\in [0,1)$ and $y=\ga(1)$.
Therefore there must exist a smooth curve $\si:[0,1)\to\n$ such
that $\vph(\si(s))=\ga(s)$ for all $s\in [0,1)$ 
and which converges to infinity
in the sense that for every compact subset $\Om$ of $\n$,
we have $\si(s)\notin\Om$ for $s\in [0,1)$ sufficiently 
close to $1$.
In particular, the curve $\si$ must have infinite length
with respect to $g_0$. By Definition \ref{initconddef},
for any $M>0$, the length of $\ga$ with respect to $g(t)$
must then be at least
$M$ for $t>0$ sufficiently small depending on $M$.

However, the curve $\ga$ 
has some finite length with respect to each of the metrics
$g(t)$, and by virtue of the uniform curvature bound, these lengths
are uniformly bounded above by some number $L$, say
(see for example \cite[Lemma 5.3.2]{RFnotes}).
This is a contradiction, and we have concluded that $\vph$
must be a diffeomorphism.

The fact that $g(t)$ can be extended smoothly down to $t=0$
then follows directly from Definition \ref{initconddef}.
\end{proof}

\begin{proof} (Corollary \ref{uniquenesscor}.)
By Theorem \ref{noncollapsedthm}, both of the Ricci flows
$g_1(t)$ and $g_2(t)$ can be extended to $t=0$ and then
have uniformly bounded curvature over some nonempty time 
interval $[0,\de]$.
If we let $\vph_1:\n\to\m_1$ and $\vph_2:\n\to\m_2$
be the maps from Definition \ref{initconddef} corresponding to
$g_1(t)$ and $g_2(t)$ respectively -- which are diffeomorphisms
in this case -- then $g_0=\vph_1^*(g_1(0))=\vph_2^*(g_2(0))$,
and so $\psi:=\vph_2\circ (\vph_1^{-1})$ is an isometry
from $(\m_1,g_1(0))$ to $(\m_2,g_2(0))$.
Thus $g_1(t)$ and $\psi^*(g_2(t))$ are both complete 
bounded-curvature Ricci flows, for $t\in [0,\de]$, 
which agree at $t=0$ and are thus identical by the uniqueness
result of Chen-Zhu \cite{chenzhu}, or (more simply in this
two-dimensional situation) by the uniqueness implied by
\cite[Theorem 4.2]{GT1}.
\end{proof}

\section{Flows contracting cusps}
\label{fccsect}

\subsection{Metrics on the punctured disc}
\label{PDsect}

We will require some asymptotic information about metrics
on the two-dimensional punctured disc 
$D\backslash\{0\}$ which are complete with negative 
curvature near the puncture.

We will be working on $D\backslash\{0\}$ either with respect
to the standard complex coordinate $z=x+iy$, sometimes appealing
to the corresponding standard polar coordinates $(r,\th)$, 
or with respect to the cylindrical coordinates $(s,\th)$,
where $s=-\ln r$. Note that $(s,\th)$ coordinates
are conformally equivalent to the original $(x,y)$ coordinates,
and changing coordinates $(x,y)$ to $(s,\th)$ changes the 
conformal factor according to
\begin{equation}
\label{changeofconffactor}
|\dz|^2=dx^2+dy^2=r^2(ds^2+d\th^2).
\end{equation}

With this notation, the complete conformal 
hyperbolic metric on $D\backslash\{0\}$
can be written $e^{2v}(ds^2+d\th^2)$ where $v=-\ln s$ (for
$s>0$).

\cmt{check $H$ below because we corrected it}

\begin{lemma}
\label{upperbdlemma}
If $g_0=e^{2a}|\dz|^2$ is any smooth conformal metric on the punctured
disc $D\backslash\{0\}$ with Gauss curvature bounded above by 
$-1$ (with $g_0$ not necessarily complete) and 
$H=[r\ln r]^{-2}|\dz|^2$ is the complete conformal hyperbolic metric on
$D\backslash\{0\}$, then $g_0\leq H$, or equivalently 
\begin{equation}
\label{firstabd}
a\leq -ln[r(-\ln r)].
\end{equation}
Moreover, if $g(t)=e^{2u(t)}|\dz|^2$ is any smooth Ricci flow on
$D$ ($t\in [0,T]$) with $g(0)\leq g_0$ 
then 
\begin{equation}
\label{RFupperbd}
u\leq -ln[r(-\ln r)]+\half\ln (1+2t).
\end{equation}
\end{lemma}

\cmt{we're in middle of changing notation from $a$ to $u$ 
since we're
working on the whole disc in the second part.
Not yet introduced $u_0$.}

\begin{proof}
With respect to $(s,\th)$ coordinates as introduced at the 
start of Section \ref{PDsect}, the conformal factor
$$v_0:=-\ln s$$
gives rise to the complete hyperbolic metric on $(0,\infty)\times S^1$.
Moreover, for $\de>0$ the conformal factor
$$v_\de:=-\ln\left[\frac{\sin(\de(s-\de))}{\de}\right]$$
defines the complete hyperbolic metric over the range 
$s\in I_\de:=(\de,\frac{\pi}{\de}+\de)$.
It is elementary to see that this conformal factor must
be pointwise at least as large as the conformal factor $w$ of any
other conformal metric on $\overline{I_\de}\times S^1$ 
with Gauss curvature
no higher than $-1$.
Indeed, for sufficiently large $0<M<\infty$, we must have
$v_\de+M>w$ (since the right-hand side is bounded and $v_\de$
is bounded below) 
and then we can reduce $M>0$ continuously 
without this condition failing until possibly at $M=0$
since if it suddenly failed for $M>0$ at some point $p$,
then $v_\de-w+M$ would be a weakly positive function with
a zero at $p$ but with strictly negative Laplacian at $p$:
\begin{equation*}
\begin{aligned}
\lap (v_\de-w+M)&=-e^{2v_\de}K[e^{2v_\de}|\dz|^2]
+e^{2w}K[e^{2w}|\dz|^2]\\
& \leq  e^{2v_\de}-e^{2w}=e^{2w}(e^{-2M}-1)\\
&<0
\end{aligned}
\end{equation*}
which is a contradiction.

Thus 
$$w\leq v_\de\to v_0 \qquad\text{ as }\de\downto 0,$$
and returning from $(s,\th)$ to $(x,y)$ coordinates, keeping
in mind \eqref{changeofconffactor}, we deduce the first part
\eqref{firstabd} of the lemma.

For the second part of the lemma, note that 
the function $v_\de+\half\ln (1+2t)$ is the conformal 
factor of a Ricci flow on $I_\de\times S^1$
which starts at $t=0$ above $v_0$, and hence above any
conformal factor $w$ as above. 
By the maximum principle, $v_\de+\half\ln (1+2t)$ must then
lie above any conformal factor on $D$ which represents a 
Ricci flow and which starts below $w$ at $t=0$.
Letting $\de\downto 0$ then yields \eqref{RFupperbd}.
\end{proof}

We now turn to the subtler issue of lower bounds for 
conformal factors of metrics $g_0$ as in Lemma \ref{upperbdlemma}.

\begin{lemma}
\label{lowerbdlemma}
Suppose $g_0=e^{2a}|\dz|^2$ is a smooth conformal metric on the 
punctured disc $D\backslash\{0\}$ with Gauss curvature bounded 
within some interval $[-M,-1]$ and with $g_0$ complete at the origin.
Denoting the complete conformal hyperbolic metric on
$D\backslash\{0\}$ by $H=e^{2v}|\dz|^2$,
where $v=-\ln [-r\ln r]$ as above, we have
\begin{equation}
\label{lowerabd}
a-v\geq -C
\end{equation}
for some $C<\infty$ (depending on $g_0$) and any $r\in (0,\half)$,
and in particular, $a\to \infty$ as $r\downto 0$.
\end{lemma}

To clarify, by \emph{complete at the origin} we mean that $g_0$
restricted to, say, $\overline{D_\half}\backslash\{0\}$ should
be a complete manifold with boundary.

\begin{proof} (cf. \cite{GT2}.)
Choose any cut-off function $\vph\in C_c^\infty(D_\frac{3}{4},[0,1])$
with $\vph\equiv 1$ on $D_\half$, and consider the metric
$\Om=e^{2\al}|\dz|^2$ defined by
$$\al=\vph a+(1-\vph)v.$$
For $r\in (\frac{3}{4},1)$, we have $\Om=H$, and so $K[\Om]=-1$.
For $r\in (0,\half)$, we have $\Om=g_0$, and so $K[\Om]\geq -M$.
In the remaining \emph{compact} region $\half \leq r \leq \frac{3}{4}$,
the curvature $K[\Om]$ has \emph{some} lower bound, and thus there 
exists $\be\leq \infty$ such that 
$$K[\Om]\geq -\be$$
throughout $D\backslash\{0\}$. Since $\Om$ is clearly complete,
we may apply Yau's Schwarz lemma (see \cite{Yau73} and
\cite[Theorem 2.3]{GT1}) to deduce that
$H\leq \be e^{2\al}|\dz|^2$, or equivalently
$$v\leq \half\ln\be+\al.$$
Since 
$\al=a$
on $D_\half$, the lemma is proved with 
$C=\half\ln\be$.
\end{proof}

\subsection{Spherical upper barriers}

In this section we consider Ricci flows on the disc
which begin at a metric as considered in Lemma \ref{upperbdlemma}.
The goal is to exploit the estimates from the previous 
section in order to construct an upper barrier which gives 
decay of the conformal factor like $1/t$.

\begin{lemma}
\label{flowupperbdlemma}
If $g_0=e^{2a}|\dz|^2$ is any smooth conformal metric on the punctured
disc $D\backslash\{0\}$ with Gauss curvature bounded above by 
$-1$ (with $g_0$ not necessarily complete) and 
$g(t)=e^{2u(t)}|\dz|^2$ is any smooth Ricci flow on
$D$ ($t\in [0,T]$) with $g(0)\leq g_0$ 
then there exists $\be<\infty$ universal such that
\begin{equation}
u\leq \frac{\be}{t}
\end{equation}
for $r\leq \half$ and $0<t < \min\{1,T\}$.
\end{lemma}

The function $s:\R^2\to\R$ defined in polar coordinates by
$$s(r):=\ln\frac{2}{1+r^2}$$
gives rise to the metric $e^{2s}|\dz|^2$ of the 
round (punctured) sphere.
One may also dilate this conformal factor to $s(\frac{r}{\lambda})$,
or add a constant, giving another spherical metric of a possibly
different curvature.
Under Ricci flow such a  conformal factor evolves simply by
shifting downwards - i.e. subtracting off a time-dependent
constant. For example, for any $\lambda>0$, one Ricci flow would
be given by the conformal factor
$$(r,t)\mapsto s\left(\frac{r}{\lambda}\right)-\ln \la + \half\ln (1-2t).$$

The idea in this section is to use these spherical metrics, 
appropriately restricted, as upper barriers for the Ricci flow
$g(t)$ of the lemma. Moreover, we evolve them not just by
Ricci flow (i.e. subtracting off a time-dependent constant)
but also by dilating within the domain. 
Whereas a Ricci flow would make the radius of a sphere
shrink like $\sqrt{C-t}$, our barriers will have a radius 
which is \emph{increasing} like $t$.

One difficulty with this approach is that one must take care
in any maximum principle argument about what is happening on
the boundary of the  domain on which one is working. This is
where the estimates of the previous section first come in.

\begin{proof} (Lemma \ref{flowupperbdlemma}.)
Without loss of generality, we may assume that $T\leq 1$.
Then by Lemma \ref{upperbdlemma}, for $t\in (0,T)$, 
we have the upper bound
$$u(r,t)< h(r):=-\ln[r(-\ln r)]+\half\ln 3$$
for $r\in (0,1)$.
On the other hand, we consider the function
$S:D\times (0,T)\to\R$ defined by
$$S(r,t):=s\left(\frac{r}{\lambda}\right)
-\ln [\la(-\ln \la)] + \half\ln 3$$
where $\la=\la(t):= e^{-\frac{6}{t}}$ will be motivated in a moment.
As mentioned above, $S(\cdot,t)$ represents
the conformal factor of part of some sphere for each $t$.
Note that
$$S(\la,t)=h(\la),$$
and so we can define a continuous function $U:D\times (0,T)\to\R$ by
$$U(r,t)=\left\{
\begin{aligned} 
S(r,t) & \qquad 0\leq r < \la \\ 
h(r) & \qquad \la \leq r<1
\end{aligned}.
\right.
$$
{\bf Claim:} On the whole of $D\times (0,T)\to\R$ we have
$$u\leq U$$

\emph{Proof of Claim:}
For $0\leq r\leq \la$, we have 
$S(r,t)\geq -\ln [\la(-\ln \la)]\to \infty$ as $t\downto 0$.
Therefore for sufficiently small $t>0$ (depending on the flow
in question) we must have $u(r,t) < U(r,t)$ for all 
$r\in [0,1)$.

Now suppose at some first time $t_0\in (0,T)$ the function 
$U(\cdot,t_0)$ fails to be a strict upper barrier for $u(\cdot,t_0)$.
Then we can find $r_0\in (0,\la(t_0))$ 
such that $U(r_0,t_0)=u(r_0,t_0)$
even though $U(\cdot,t_0)\geq u(\cdot,t_0)$.
At $(r_0,t_0)$ we then have
$$\pl{(U-u)}{t}\leq 0;\qquad \lap (U-u)\geq 0,$$
but the Ricci flow equation \eqref{2DRFeq}
gives $\pl{u}{t}=e^{-2u}\lap u$,
and so 
\begin{equation}
\label{Uoneway}
\pl{U}{t}\leq e^{-2U}\lap U
\end{equation}
at $(r_0,t_0)$.
On the other hand, keeping in mind that $\la=e^{-\frac{6}{t}}$,
we have at $(r_0,t_0)$ that
\begin{equation*}
\begin{aligned}
\pl{U}{t}&=\pl{S}{t}=
\left[-\frac{r_0}{\la^2}s'\left(\frac{r_0}{\la}\right)
-\frac{1}{\la}-\frac{1}{\la\ln\la}\right]\frac{d\la}{dt}\\
&\geq -\frac{1}{\la}\frac{d\la}{dt}
= -\frac{6}{t^2}
\end{aligned}
\end{equation*}
and
$$e^{-2U}\lap U=e^{-2S}\lap S=-\frac{1}{3}(\ln \la)^2=
-\frac{12}{t^2},$$
and so
$$\pl{U}{t}- e^{-2U}\lap U\geq \frac{6}{t^2}>0,$$
contradicting \eqref{Uoneway} and proving the claim.

By inspection, the maximum of $U$ for $r\leq \half$
is achieved at the origin ($r=0$):
$$\sup_{r\leq \half}U=s(0)-\ln [\la(-\ln \la)] + \half\ln 3
=\ln 2+\frac{6}{t}-\ln(\frac{6}{t})+\half\ln 3
\leq \frac{C}{t}$$
for some universal $C$, since $t<1$.
By the claim above, we then have for $r\leq \half$ and 
$t\in (0,T)$
$$u\leq \frac{C}{t}$$
as desired.
\end{proof}


\subsection{Truncating cusps}

In this section we clarify how to take a cusp-like
metric on a punctured surface and smooth it out across the puncture
in a controlled manner.

\begin{lemma}
\label{trunclemma}
Suppose $g_0=e^{2a}|\dz|^2$ is a smooth complete 
conformal metric on the 
punctured closed disc $\overline{D}\backslash\{0\}$ 
with Gauss curvature bounded 
within some interval $[-M,-1]$.
Then there exists an increasing sequence  of smooth conformal metrics
$g_k=e^{2u_k}|\dz|^2$ on $D$ such that 
\begin{compactenum}[(i)]
\item
$g_k=g_0$ on $D\backslash D_{1/k}$;
\item
$g_k\leq g_0$ throughout $D\backslash\{0\}$;
\item
$\inf_{D_{1/k}}u_k \to\infty$ as $k\to\infty$;
\item
the Gauss curvatures of $g_k$ are uniformly bounded below
independently of $k$.
\end{compactenum}
\end{lemma}

\begin{proof}
Pick any smooth function $\psi:\R\to\R$ such that
\begin{compactenum}[(a)]
\item
$\psi(s)=s$ for $s\leq -1$;
\item
$\psi(s)=0$ for $s\geq 1$;
\item
$\psi''\leq 0$.
\end{compactenum}
The lemma is proved by taking 
a subsequence of the increasing sequence of metrics $g_k$ defined by
their conformal factors
$$u_k:=\psi(a-k)+k.$$
One should view these metrics as a smoothed out version of
the minimum of $g_0$ and the metric with constant conformal factor $k$.
Indeed, it is clear that $(ii)$ is satisfied.
Note that $u_k$ is identically equal to $k$ in a neighbourhood
of the origin because $a\to\infty$ at the origin,
according to Lemma \ref{lowerbdlemma}, and therefore the metrics
extend smoothly across the origin.
Assertion $(iii)$ also follows because $a\to\infty$ at the origin.

In order to see that the Gauss curvature of the metrics $g_k$
is uniformly controlled from below, we divide up into three cases:
\begin{compactenum}[(a)]
\item
Where $a\geq k+1$, we have $u_k\equiv k$, so $K(g_k)=0$;
\item
Where $a\leq k-1$, we have $u_k\equiv a$, so $K(g_k)=K(g_0)\geq -M$;
\item
Where $k-1< a<k+1$, we have $u_k\geq a-1$. We compute
$$\lap u_k=\psi'(a-k)\lap a + \psi''(a-k)|\grad a|^2\leq 
\psi'(a-k)\lap a\leq \lap a,$$
since $\psi''\leq 0$, $\psi'\in [0,1]$ and 
$\lap a=-e^{2a}K(g_0)\geq 0$. Therefore
$$K(g_k)=-e^{-2u_k}\lap u_k\geq -e^{-2u_k}\lap a
\geq -e^{-2(a-1)}\lap a=e^2K(g_0)\geq -e^2M.$$
\end{compactenum}
Finally, Assertion $(i)$ can be guaranteed after passing to
an appropriate subsequence.
\end{proof}

\subsection{Proof of Theorem \ref{cuspcontractthm}}

We now combine the supporting results we have compiled in
Section \ref{fccsect}
into a proof of the existence of the Ricci flows claimed
in Theorem \ref{cuspcontractthm}.

For simplicity of notation, we restrict our discussion
to the case that $n=1$ -- that is, there is a single 
puncture $p$ on $\m$ -- although the proof of the general
case will then be an obvious extension.

We begin by taking isothermal coordinates $x$ and $y$ in 
a neighbourhood of $p$. By scaling and translating them, 
we may assume that $p$ corresponds to $x=y=0$, that
the coordinates exist for $z=x+iy$ within a domain containing 
the closure of the unit disc $D\subset\C$,
and that the supremum of the curvature of $g_0$ in 
$D\backslash \{0\}$ is 
strictly negative.
By scaling the metric itself, we may then assume without loss
of generality that the
Gauss curvature of $g_0$ is less than $-1$ within $D\backslash \{0\}$.

By truncating $g_0$ within $D$ using Lemma \ref{trunclemma}, 
we can find an increasing sequence  of smooth conformal metrics
$g_k$ on $\m$ such that 
\begin{compactenum}[(i)]
\item
$g_k=g_0$ on $\m\backslash\{p\}$ 
outside the shrinking neighbourhood $D_{1/k}$ of $p$;
\item
$g_k\leq g_0$ throughout $\m\backslash\{p\}$;
\item
near $p$, the conformal factors $u_k$ of $g_k$ satisfy 
$\inf_{D_{1/k}}u_k \to\infty$ as $k\to\infty$;
\item
the Gauss curvatures of $g_k$ are uniformly bounded below
independently of $k$.
\end{compactenum}
We now flow each of the smooth metrics $g_k$ under Ricci flow
in order to give a time-dependent flow $g_k(t)$. Ricci flow 
theory in two dimensions due to Hamilton and Chow \cite{chowknopf}
tells us that the flows exist for all time if the genus of
$\m$ is at least one, while if the genus is zero, then 
the existence time is equal to the area of $(\m,g_k)$ divided
by $8\pi$. In particular, the existence time is increasing with
$k$ since the areas of $(\m,g_k)$ are increasing with $k$, and
we may pick some uniform 
$T>0$ so that all of these flows exist for $t\in [0,T]$.

The maximum principle applied to conformal factors
tells us that because the metrics $g_k(0)$
are increasing with $k$, so are the metrics $g_k(t)$ for each 
$t\in [0,T]$. Also, the maximum principle applied to 
curvatures, and condition (iv) above, tell us that the Gauss curvature of the flows $g_k(t)$ is uniformly bounded below
independently of $k$ and $t$.

We also want to consider Shi's complete bounded-curvature 
Ricci flow $g_{s}(t)$ on $\m\backslash\{p\}$ with $g_{s}(0)=g_0$.
Within $D\backslash \{0\}$, the conformal factor of $g_0$ 
is converging to infinity
at the puncture (this is implicit above and follows from 
Lemma \ref{lowerbdlemma}). 
Therefore, working directly from the Ricci flow equation
\eqref{2DRFeq}
and using the fact that the curvature is bounded above, we see
that the conformal factor of each of the metrics $g_s(t)$
must also converge to infinity at the puncture.
This allows us to apply the maximum principle to compare 
each $g_k(t)$ with $g_s(t)$, and we conclude that 
for each $t\in [0,T]$,
\begin{equation}
\label{gkgs}
g_k(t)\leq g_{k+1}(t)\leq g_s(t)
\end{equation}
throughout $\m\backslash \{p\}$.

This estimate gives good control on the approximating flows
$g_k(t)$ away from $p$, all the way down to $t=0$. 
In particular, it allows us to define a flow 
\begin{equation}
\label{Gdef}
G(t)=\lim_{k\to\infty} g_k(t)
\end{equation}
on $\m\backslash\{p\}$ for $t\in [0,T]$, and when we take any conformal
chart not containing $p$, we see that the conformal factors
of $g_k(t)$ are locally uniformly bounded, independently of $k$
(since $g_k(t)$ is sandwiched between $g_1(t)$ and $g_s(t)$)
and so we may apply standard parabolic regularity theory
to get local uniform bounds on their derivatives.
Therefore, we deduce that $G(t)$ is a smooth Ricci flow
on $\m\backslash\{p\}$ for $t\in [0,T]$.

We propose that $G(t)$ extends to be the flow whose existence
is asserted in the theorem.
Certainly $G(t)\to g_0$ locally on $\m\backslash\{p\}$ as $t\downto 0$.
However, at this point it is unclear whether for $t>0$, 
$G(t)$ extends smoothly
across $p$. Indeed, the truncated cusps
within the metrics $g_k(t)$ might take longer and longer to
contract as $k\to\infty$, and $G(t)$ might then 
coincide with Shi's flow $g_s(t)$, for example. 
In order to show that the truncated cusps
within the metrics $g_k(t)$ contract at a rate which is
independent of $k$, we apply Lemma 
\ref{flowupperbdlemma} to each flow $g_k(t)$ 
restricted to $D$.

This gives us uniform control from above on the metrics
$g_k(t)$ throughout $\m$, on any closed time interval 
within $(0,T]$, independently of $k$.
Therefore for $t>0$, we can extend the definition
\eqref{Gdef} to the whole of $\m$. Moreover, this uniform
control allows us to apply the same standard parabolic regularity 
theory as above to get local uniform bounds on the derivatives
of $g_k(t)$ across $p$, and we conclude that $G(t)$ is
a smooth Ricci flow on the whole of $\m$ for $t>0$.

We now see that the Ricci flow $(\m,G(t))$ for $t\in (0,T]$
has $(\m\backslash\{p\},g_0)$ as initial metric in the 
sense of Definition \ref{initconddef}, with the map $\vph$ of the
definition equal to the obvious inclusion.

This completes the proof of the first part of the theorem.
It remains to argue that the diameter of $(\m,G(t))$ decays
precisely logarithmically in $t$ as in \eqref{logdecay}.
Since the flow is only singular at $p$, it suffices to argue that 
the distance from $p$ ($z=0$ in the chart considered above) 
to any other fixed point in $\m$ (say the point $z=1/2$) decays
in this logarithmic fashion.

\emph{Upper bound:} Because the flow $G(t)$ above arose as a limit
of flows $g_k(t)$, it suffices to prove a logarithmic upper bound
for the distance from $z=0$ to $z=1/2$ in the flows $(\m,g_k(t))$ provided that it is independent of $k$.
Taking $\be$ from Lemma \ref{flowupperbdlemma}, and exploiting
that result together with Lemma \ref{upperbdlemma}
we compute
\begin{equation}
\begin{aligned}
dist_{g_k(t)}(z=0,z=1/2) &\leq  \int_0^\half e^{u_k(s,t)}ds\\
&= \int_0^{e^{-\be/t}} e^{u_k(s,t)}ds
+\int_{e^{-\be/t}}^\half e^{u_k(s,t)}ds\\
&\leq \int_0^{e^{-\be/t}} e^{\be/t}ds
+\int_{e^{-\be/t}}^\half \frac{\sqrt{3}}{s(-\ln s)}ds\\
&=1+\sqrt{3}\bigg[-\ln(-\ln s)\bigg]_{e^{-\be/t}}^\half
\end{aligned}
\end{equation}
for $t\in(0,\min\{1,T\})$ sufficiently small.
Therefore, for sufficiently small $t\in (0,\min\{1,T\})$
we have
$$dist_{g_k(t)}(z=0,z=1/2) \leq -2\ln t$$
as desired.

\emph{Lower bound:} 
For points near the origin, we will derive control on the curvature
for a certain time depending on the proximity to the origin.
This will then show that we cannot deviate from the original metric
too much for a controlled time (again, depending on the 
proximity to the origin) and will lead to a lower bound
for the diameter.

{\bf Claim:}
There exists $C<\infty$ such that for $r>0$ sufficiently small, 
on $\partial D_r$, the conformal factor of the Ricci flow
is bounded below
by $-C-\ln[r(-\ln r)]$ for a time $\frac{1}{C(\ln r)^2}$.

Throughout the argument, $C$ will denote a positive constant which may
get larger each time it is used.

Before proving the claim, we show how it would imply the 
desired lower bound on the diameter.
Suppose that the claim holds for  $r\in (0,\overline r)$.
Then for $t>0$ sufficiently small, we would have the 
conformal factor at time $t$ bounded below by 
$-C-\ln[r(-\ln r)]$ for 
$r\in (\underline r, \overline r)$, where 
$\underline r := e^{-\frac{1}{C\sqrt{t}}}$.
Then the distance between $\partial D_{\underline r}$
and $\partial D_{\overline r}$ with respect to the time
$t$ metric must be at least
$$\int_{\underline r}^{\overline r} e^{-C-\ln[r(-\ln r)]}dr
\geq \frac{1}{C}\bigg[-\ln(-\ln r)\bigg]_{\underline r}^{\overline r}\geq \frac{1}{C}\ln (\frac{1}{C\sqrt{t}})\geq 
-\frac{1}{C}\ln t$$
for sufficiently small $t$, as desired.

\emph{Proof of claim:}
Note that the lower bound of the claim holds at $t=0$ by Lemma
\ref{lowerbdlemma}. By inspection of the Ricci flow equation
\eqref{2DRFeq},
the claim will follow if we can bound the Gauss curvature
on $\partial D_r$ by $C(\ln r)^2$ for a time $\frac{1}{C(\ln r)^2}$
(for sufficiently small $r>0$).

Consider the hyperbolic metric $H$ on $D\backslash\{0\}$
defined in terms of its conformal factor
$h(z)=-\ln(-|z|\ln |z|)$.
As $z\in D\backslash\{0\}$ approaches the origin,
the injectivity radius at $z$ with respect to $H$ is 
asymptotically $\frac{\pi}{-\ln|z|}$.
Therefore, for $z$ sufficiently close to the origin, we have
the volume ratio bound
$$\frac{\Vol_H(B_H(z,s))}{\pi s^2}\geq 1$$
for $s\in (0,\frac{\pi}{-2\ln|z|})$, say.

By virtue of Lemmata \ref{upperbdlemma} and \ref{lowerbdlemma},
we see that $g_0$ is equivalent to $H$ on 
$D_\half\backslash\{0\}$, say, and thus for
$z$ sufficiently close to the origin and $r_0=\frac1{-\ln|z|}$,
we have
$$\frac{\Vol_{g_0}(B_{g_0}(z,r_0))}{r_0^2}\geq \frac1C.$$
By applying Proposition \ref{chenprop} to the flow
$G(t)$ (or strictly speaking to $G(t+\ep)$ for arbitrarily small
$\ep>0$) on $D$ with $x_0\in D\backslash\{0\}$ sufficiently 
close to $0$ and $r_0=\frac1{-\ln|x_0|}$, we deduce the
Gauss curvature control
$$|K|(x_0)\leq C(\ln|x_0|)^2 \text{ for }
t\leq \frac{1}{C(\ln|x_0|)^2}$$
as required to complete the proof.

We remark that a by-product of the argument we have just given
is that the supremum of the conformal factor at small time $t>0$
is bounded below by $\frac{1}{C\sqrt{t}}$.
This can be compared to the upper bound $\frac Ct$
implied by Lemma \ref{flowupperbdlemma}.

\section{Alternative uniqueness issues}
\label{alt_sect}

The example of a contracting cusp that we have constructed
in Section \ref{fccsect}
demonstrates that Ricci flows are nonunique when posed as
in Definition \ref{initconddef}.
However, one can also ask whether our newly constructed
flows are unique amongst all flows which contract their
cusps, and in Theorem \ref{unique_after_all} we asserted that
they are. This section is devoted to proving that assertion.

The essential difficulty is that \emph{a priori} we
know little about the behaviour of any competitor flow
near the punctures, for small time.

Recall that $\m$ is a compact Riemann surface and 
$\{p_1,\ldots,p_n\}\subset \m$ 
is a finite set of distinct points. We have a 
complete bounded-curvature smooth conformal metric 
$g_0$ on $\n:=\m\backslash \{p_1,\ldots,p_n\}$
with strictly negative curvature in a neighbourhood of each point
$p_i$, and (from Theorem \ref{cuspcontractthm})
a complete Ricci flow $g(t)$ on $\m$
for $t\in (0,T]$ (for some $T>0$) with curvature uniformly 
bounded below, and such that 
$$g(t)\to g_0$$
smoothly locally on $\n$ as $t\downto 0$.

\cmt{do we not want to unify the notation between claims 1 and 2?
Basically, we can't use the notation of the theorem because 
we will need to switch the 2 solns to prove the thm}

{\bf Claim 1:} With 
$\tilde g(t)$ any Ricci flow as in Theorem \ref{unique_after_all},
(that is, defined on $\m$ for $t\in (0,\de)$,
with curvature bounded below and satisfying
$\tilde g(t)\to g_0$ smoothly locally on $\n$ as $t\downto 0$)
if $\si(t)$ is any Ricci flow on $\m$ for $t\in [0,\de)$
such that $\si(0)< g_0$
on $\n$, then $\si(t)\leq \tilde g(t)$ on $\m$ for $t\in (0,\de)$.

We will use Claim 1 to prove:

{\bf Claim 2:} Given two such flows $\tilde g_1(t)$ and 
$\tilde g_2(t)$ (that is, defined on $\m$ for $t\in (0,\de)$, with
curvature bounded below and converging to $g_0$
smoothly locally on $\n$ as $t\downto 0$) 
we must have $\tilde g_1(t)\leq \tilde g_2(t)$ for $t\in (0,\de)$.

Once we have established Claim 2, by switching $\tilde g_1(t)$ and 
$\tilde g_2(t)$ we will have $\tilde g_1(t)=\tilde g_2(t)$,
and by applying this in the case $\tilde g_1(t)=\tilde g(t)$
and $\tilde g_2(t)=g(t)$, we will have
finished the proof of Theorem \ref{unique_after_all}.

To prove Claim 2 from Claim 1, we will consider a scaling of the flow 
$\tilde g_1(t)$ starting at some early time $t_0>0$.
By the lower curvature bound assumption 
$K[\tilde g_1(t)]\geq -\be\leq 0$ say, we
have $e^{-\be t_0}\tilde g_1(t_0)\leq g_0$ 
and better still, $e^{-2\be t_0}\tilde g_1(t_0)<g_0$.
Therefore, the Ricci flow 
$\si(t):=e^{-2\be t_0}\tilde g_1(e^{2\be t_0}t+t_0)$
considered for $t\in [0,(\de-t_0)e^{-2\be t_0})$
satisfies the hypotheses of Claim 1 (with $\tilde g(t)$ there
equal to $\tilde g_2(t)$ here) and so
$$e^{-2\be t_0}\tilde g_1(e^{2\be t_0}t+t_0)\leq \tilde g_2(t)$$
for $t\in [0,(\de-t_0)e^{-2\be t_0})$.
Taking the limit $t_0\downto 0$ then finishes the proof of Claim 2.

It remains to prove Claim 1, and for that we need some 
\emph{a priori} control on solutions.
The key ingredient is:

{\bf Claim 3:} Take any flow $\tilde g(t)$ as above
(that is, defined on $\m$ for $t\in (0,\de)$,
with curvature bounded below and satisfying
$\tilde g(t)\to g_0$ smoothly locally on $\n$ as $t\downto 0$).
Choose
a local complex coordinate $z$ about one of the punctures $p_i$,
and write $\tilde g(t)$ locally as $e^{2u}|\dz|^2$.
Then for any $M<\infty$, we have
$$u\geq M$$
in some neighbourhood of $p_i$, for sufficiently small $t>0$.

To prove Claim 1 from Claim 3, look at a neighbourhood of a point
$p_i$ as in Claim 3. Denote the conformal factor of $\si(0)$ in
this local chart by $s$ (i.e. $\si(0)=e^{2s}|\dz|^2$) 
and define $M<\infty$ to be the
supremum of $s$ over some neighbourhood of $p_i$. By Claim 3,
we may shrink this neighbourhood and be sure that $u\geq M\geq s$
for sufficiently small $t>0$. Repeating for all the other punctures
$p_i$, 
we can find an open set $\Om\subset \m$ containing each point
$p_i$ so that $\tilde g(t)\geq \si(0)$ on $\Om$ for
$t\in (0,t_0)$ (for some $t_0\in (0,\de)$). 
By compactness of $\m\backslash\Om$ and the 
fact that $\si(0)< g_0$, we may reduce $t_0>0$ further and 
be sure that $\tilde g(t)\geq \si(0)$ throughout the whole of $\m$
for $t\in (0,t_0)$. The comparison principle then tells us that
for any $\tilde t\in (0,t_0)$, we have
$\tilde g(t+\tilde t)\geq \si(t)$ for $t\in (0,\de-\tilde t)$.
By taking the limit $\tilde t\downto 0$, Claim 1 is proved.

It now remains to prove Claim 3, and this in turn relies on
a new claim:

{\bf Claim 4:} In the setting of Claim 3, if $\Om$ is a 
neighbourhood of $p_i$ compactly contained in the neighbourhood
where $z$ is defined, then there exists \emph{some} $m\in\R$ such that
$$u\geq m$$
within $\Om$ for $t\in (0,\de/2]$.

To prove Claim 4, we use the fact that the Gauss curvature is 
uniformly bounded below by $-\be$, say.
Then by the Ricci flow equation \eqref{2DRFeq}, for $z\in\Om$,
and $t\in (0,\de/2]$,
$$u(z,t)\geq u(z,\de/2)-\be (\de/2-t)\geq 
\inf_\Om u(\cdot,\de/2)-\be \de/2=:m,$$
and the claim is proved.

Finally, we prove Claim 3 from Claim 4.
With respect to the local complex coordinate from Claim 3,
we may write $g_0$ as $e^{2u_0}|\dz|^2$.
Recalling that the curvature of $g_0$ is uniformly strictly negative
in some neighbourhood of $p_i$, we see that without loss of
generality, we may assume that the coordinate $z$ is defined
for $z\in D$, and that for $z\in D\backslash\{0\}$
the curvature of $g_0$ lies in some interval $[-C,-1]$.
By Lemma \ref{lowerbdlemma}, we then see that 
the conformal factor $u_0$ of $g_0$
is at least $M+2$ for $z$ in some small disc $D_{\ep}$ within 
the conformal chart (away from the origin).

Consider 
$$(M-u)_+:=\left\{\begin{aligned}
M-u  & \qquad\text{if }M-u>0\\
0 & \qquad\text{otherwise.}
\end{aligned}
\right.$$

Using the fact that $u\to u_0$ smoothly locally on 
$\overline{D_\ep}\backslash\{0\}$, together with 
Claim 4, we see that $\|(M-u)_+\|_{L^1(D_\ep)}\to 0$ as $t\downto 0$.
Moreover, because $u_0\geq M+2$ on $\partial D_\ep$, we see that 
$u\geq M+1$ on $\partial D_\ep$ over some nonempty time interval 
$(0,t_0)$.
Pick any smooth function $\phi:\R\to\R$ such that
\begin{compactenum}[(a)]
\item
$\phi(s)=s$ for $s\geq 1$;
\item
$\phi(s)=0$ for $s\leq -1$;
\item
$\phi''\geq 0$,
\end{compactenum}
and note that $\phi'\geq 0$.
We then have $\phi(M-u)=0$ on $\partial D_\ep$ over 
the time interval $(0,t_0)$ and 
we may compute for $t\in (0,t_0)$,
\beq
\begin{aligned}
\frac{d}{dt}\int_{D_\ep}\phi(M-u)&=-\int \phi'(M-u)u_t
=-\int \phi'(M-u)e^{-2u}\lap u\\
&= -\int e^{-2u}|\grad u|^2\left(\phi''(M-u)+2\phi'(M-u)\right)\\
&\leq 0.
\end{aligned}
\eeq
By allowing $\phi$ to decrease uniformly to the function 
$s\mapsto s_+$, we then see that
$$(M-u)_+\equiv 0$$
on $D_\ep$ for $t\in (0,t_0)$, completing the proof.

{\sc mathematics institute, university of warwick, coventry, CV4 7AL,
uk}\\
\url{http://www.warwick.ac.uk/~maseq}
\end{document}